\documentclass{amsart}

\usepackage{macros}
\standardsettings
\colorcommentstrue

\draftfalse

\renewcommand{\int}{\text{int}}
\newcommand{\ext}{\text{ext}}

\begin{document}

\authorlior\authordavid

\title{Extrinsic Diophantine approximation on manifolds and fractals}

\begin{abstract}
Fix $d\in\mathbf N$, and let $S\subseteq\amsbb R^d$ be either a real-analytic manifold or the limit set of an iterated function system (for example, $S$ could be the Cantor set or the von Koch snowflake). An \emph{extrinsic} Diophantine approximation to a point $\mathbf x\in S$ is a rational point $\mathbf p/q$ close to $\mathbf x$ which lies \emph{outside} of $S$. These approximations correspond to a question asked by K. Mahler ('84) regarding the Cantor set. Our main result is an extrinsic analogue of Dirichlet's theorem. Specifically, we prove that if $S$ does not contain a line segment, then for every $\mathbf x\in S\setminus\amsbb Q^d$, there exists $C > 0$ such that infinitely many vectors $\mathbf p/q\in \amsbb Q^d\setminus S$ satisfy $\|\mathbf x - \mathbf p/q\| < C/q^{(d + 1)/d}$. As this formula agrees with Dirichlet's theorem in $\amsbb R^d$ up to a multiplicative constant, one concludes that the set of rational approximants to points in $S$ which lie outside of $S$ is  large. Furthermore, we deduce extrinsic analogues of the Jarn\'ik--Schmidt and Khinchin theorems from known results.
\end{abstract}


\maketitle

\section{Introduction}

\ignore{
In 1984, K. Mahler published a paper entitled ``Some suggestions for further research'' \cite{Mahler}, in which he writes the following moving statement: ``At the age of 80 I cannot expect to do much more mathematics. I may however state a number of questions where perhaps further research might lead to interesting results''. One of these questions was regarding Diophantine approximation on the Cantor set.\footnote{In this paper, the phrase ``Cantor set'' always refers to the ternary Cantor set.} In Mahler's words, ``How close can irrational elements of Cantor's set be approximated by rational numbers

\begin{enumerate}
\item In Cantor's set, and
\item By rational numbers not in Cantor's set?''
\end{enumerate}
We shall refer to approximations of the first kind as \emph{intrinsic}, and to those of the second kind as \emph{extrinsic}. Also interesting is the case where the approximations may come from either inside or outside the Cantor set; in this case the approximations will be called \emph{classical}. More generally, given $d\in\N$ and a set $S\subset\R^d$, one may consider intrinsic, extrinsic, and classical approximations to points in $S$.
}

Fix $d\in\N$ and a set $S\subset\R^d$. One may divide the set of rational points into two disjoint classes: the class of rational points which lie on $S$, and the class of rational points which lie outside of $S$. Approximating points in $S$ by rational points in $S$ is called \emph{intrinsic} approximation, while approximating points in $S$ by rational points outside of $S$ is called \emph{extrinsic} approximation. More well-studied is the case where the approximations may come from either inside or outside $S$; in this case the approximations will be called \emph{ambient}.

We shall be particularly interested in two classes of sets: $S$ may be either the limit set of an iterated function system or a real-analytic manifold. Of particular prominence is the Cantor set,\footnote{In this paper, the phrase ``Cantor set'' always refers to the ternary Cantor set.} of which K. Mahler \cite{Mahler} asked: ``How close can irrational elements of Cantor's set be approximated by rational numbers (a)  In Cantor's set, and (b) By rational numbers not in Cantor's set?'' In our terminology, Mahler is asking about intrinsic and extrinsic approximation on the Cantor set, respectively. For both the limit sets of iterated function systems and for manifolds, there is already literature on both intrinsic and ambient approximation; see for example \cite{Beresnevich_Khinchin, BFKRW, FKMS, FishmanSimmons1} and the references therein. By contrast, extrinsic approximation on algebraic varieties has been studied only briefly, in \cite[Lemma 1]{DickinsonDodson}, \cite[Lemma 4.1.1]{Drutu}, and \cite[Lemma 1]{BDL}. Each of these papers proved a lemma which stated that extrinsic rational approximations to points on algebraic varieties cannot be too close to the points they approximate.

In this paper, we analyze the theory of extrinsic approximation in more detail. Our main result (Theorem \ref{theoremdirichlet}) is an extrinsic analogue of Dirichlet's theorem. We also describe results concerning extrinsic approximation which may be deduced from their intrinsic and ambient counterparts, namely analogues of the Jarn\'ik--Schmidt theorem and Khinchin's theorem.

{\bf Convention 1.} The symbols $\lesssim$, $\gtrsim$, and $\asymp$ will denote multiplicative asymptotics. For example, $A\lesssim_K B$ means that there exists a constant $C > 0$ (the \emph{implied constant}), depending only on $K$, such that $A\leq C B$. In general, dependence of the implied constant(s) on universal objects such as the set $S$ will be omitted from the notation.

\textbf{Acknowledgements.} The first-named author was supported in part by the Simons Foundation grant \#245708. The authors thank Barak Weiss for helpful suggestions.

\subsection{An extrinsic analogue of Dirichlet's theorem}
Our main theorem is as follows:

\begin{theorem}
\label{theoremdirichlet}
Fix $d\in\N$, and let $S\subset\R^d$ be either
\begin{itemize}
\item[(1)] the limit set of an iterated function system\footnote{In this paper all iterated function systems are finite and consist of similarities.} (cf. Definition \ref{definitionIFS}), or
\item[(2)] a real-analytic manifold,
\end{itemize}
and suppose that $S$ does not contain a line segment. Then for all $\xx\in S\butnot\Q^d$, there exists $C = C_\xx > 0$ such that infinitely many $\pp/q\in\Q^d\butnot S$ satisfy
\begin{equation}
\label{extrinsic}
\left\|\xx - \frac{\pp}{q}\right\| \leq \frac{C}{q^{1 + 1/d}}\cdot
\end{equation}
Here and elsewhere $\|\cdot\|$ denotes the max norm. Moreover, the function $\xx\mapsto C_\xx$ is bounded on compact sets.
\end{theorem}

We recall that (the corollary of) Dirichlet's theorem in $\R^d$ states that for all $\xx\in\R^d\butnot\Q^d$, there exist infinitely many $\pp/q\in\Q^d$ satisfying \eqref{extrinsic} with $C = 1$. Thus Theorem \ref{theoremdirichlet} says that if $S$ is as above, then for each $\xx\in S\butnot\Q^d$ there are enough extrinsic approximations to $\xx$ to re-prove Dirichlet's theorem, if one is content with a constant multiplicative error term.\footnote{An analogue of Dirichlet's theorem for which the function $\xx\mapsto C_\xx$ is unbounded was already considered in \cite[Theorem 8.1]{FKMS}. In the present case the situation is somewhat better, since the function $\xx\mapsto C_\xx$ is bounded on compact sets.} This can be contrasted with the situation for intrinsic approximation, where the best theorem that one can prove using intrinsic rationals is much worse than Dirichlet's theorem \cite[Theorem 4.3]{FKMS}.

As a concrete application of Theorem \ref{theoremdirichlet}, we consider the case where $S$ is the Cantor set, thus giving an answer to the second of Mahler's questions mentioned above. Since the Cantor set is compact and does not contain a line segment, the following is an immediate corollary of Theorem \ref{theoremdirichlet}:

\begin{corollary}
\label{corollarycantor}
Let $K$ denote the Cantor set. There exists $C > 0$ such that for each $x\in K\butnot\Q$, there exist infinitely many $p/q\in\Q\butnot K$ satisfying
\begin{equation}
\label{extrinsic1dim}
\left|x - \frac{p}{q}\right| \leq \frac{C}{q^2}\cdot
\end{equation}
\end{corollary}

\begin{remark*}
The hypotheses of Theorem \ref{theoremdirichlet} can be weakened. Recall that for $\epsilon > 0$, a subset $S$ of a metric space $X$ is said to be \emph{$\epsilon$-porous} relative to $X$ if for every ball $B(x,r)\subset X$, there exists a ball $B(y,\epsilon r)\subset B(x,r)$ which is disjoint from $S$.  Using this definition, Theorem \ref{theoremdirichlet} can be generalized as follows:
\end{remark*}

\begin{theorem}
\label{theoremdirichletv2}
Fix $d\in\N$, and let $K\subset\R^d$ be a compact set whose intersection with every line $L\subset\R^d$ is $\epsilon$-porous relative to $L$, for some $\epsilon > 0$ independent of $L$. Then there exists $C > 0$ such that for all $\xx\in K\butnot\Q^d$, there are infinitely many $\pp/q\in\Q^d\butnot K$ satisfying \eqref{extrinsic}.
\end{theorem}

Theorem \ref{theoremdirichletv2} is readily seen to be equivalent to Corollary \ref{corollarydirichlet} below. The deduction of Theorem \ref{theoremdirichlet} from Corollary \ref{corollarydirichlet} (or equivalently from Theorem \ref{theoremdirichletv2}) is given in Section \ref{sectiondirichlet}.

We remark that the condition on $K$ is satisfied whenever $K$ is the support of an Ahlfors regular measure of dimension strictly less than 1.\footnote{A measure $\mu$ on a metric space $X$ is said to be \emph{Ahlfors regular of exponent $\delta$} if there exists $C > 0$ such that for all $x\in \Supp(\mu)$ and $0 < r \leq 1$, $(1/C)r^\delta \leq \mu(B(x,r)) \leq C r^\delta$.} Moreover, if $d = 1$ the condition just reduces to $K$ itself being porous, a fairly weak geometric condition (for example it is closed under quasiconformal maps).

\subsection{The line segment hypothesis}
A key hypothesis of Theorem \ref{theoremdirichlet} is that the set $S$ does not contain a line segment.
This hypothesis is not at all automatic; there exist examples of both manifolds and fractals which contain line segments.
In the case of fractals, the Sierpinski triangle and the Sierpinski carpet are two examples of well-known fractals each of which contains a line segment. In the case of manifolds, there are numerous examples in $\R^3$ of so-called ``ruled surfaces'' which are in fact the union of lines.

In light of the above facts, one might ask whether the line segment hypothesis can be removed. However, it is necessary for the following simple reason:
\begin{observation}
If $S$ contains a rational line segment $L$, then the conclusion of Theorem \ref{theoremdirichlet} cannot hold.
\end{observation}
\begin{proof}
Fix $\xx\in L\butnot\Q^d \subset S\butnot \Q^d$. Then for all $\pp/q\in\Q^d\butnot S\subset \Q^d\butnot L$,
\[
\left\|\xx - \frac{\pp}{q}\right\| \geq \dist\left(\frac{\pp}{q},L\right) \gtrsim_L \frac{1}{q}\cdot
\]
For $q$ sufficiently large, this contradicts \eqref{extrinsic}.
\end{proof}
\noindent In particular, the Sierpinski triangle and the Sierpinski carpet mentioned above each contain the interval $[0,1]$ viewed as a subset of the $x$-axis, which is a rational line segment. Thus they cannot satisfy the conclusion of Theorem \ref{theoremdirichlet}.

It is therefore a relevant question which manifolds and fractals contain a line segment, and which do not.
In the case of manifolds, the condition can be translated into a differential condition, which we implicitly do in the proof of Claim \ref{claimrealanalytic} (cf. Remark \ref{remarkCinfty}).
In the case of fractals, the condition is somewhat harder to check directly.
On the other hand, many fractals are totally disconnected; no such fractal can contain a line segment.
To give an example of how one can check that a fractal $\Lambda$ does not contain a line segment in a case where $\Lambda$ is not totally disconnected, we demonstrate the following:

\begin{proposition}
\label{propositionvonkoch}
The von Koch snowflake curve does not contain a line segment.
\end{proposition}

It seems likely that the techniques used in the proof of Proposition \ref{propositionvonkoch} can be generalized, but it is not clear what the statement of the generalization should be.

{\bf Overview.} Sections \ref{sectionpreliminaries}-\ref{sectiondirichlet} are devoted to the proof of Theorem \ref{theoremdirichlet}, with Section \ref{sectionpreliminaries} containing preliminaries and Section \ref{sectiondirichlet} containing the body of the proof. Section \ref{sectionvonkoch} contains the proof of Proposition \ref{propositionvonkoch}. In Section \ref{sectionotherresults}, we describe some results regarding extrinsic Diophantine approximation which are corollaries of known theorems, and give some open problems.

\ignore{
We mention the Jarn\'ik--Besicovitch theorem only to state that an extrinsic analogue of the Jarn\'ik--Besicovitch theorem is not possible. Indeed, if $\psi$ decays fast enough, then algebraic considerations may prevent the existence of any  $\psi$-extrinsically approximable points. This phenomenon has already been observed in \cite[Lemma 1]{DickinsonDodson}, \cite[Lemma 4.1.1]{Drutu}, and \cite[Lemma 1]{BDL}. We state its most general form here.

\begin{theorem}
\label{theoremjarnikextrinsic}
Let $M\subset\R^d$ be the zero set of the polynomials $P_1,\ldots,P_n:\R^d\to\R$, and assume that $P_1,\ldots,P_n$ have rational coefficients. Let $D$ be the maximum of the degrees of $P_1,\ldots,P_n$, and let $\psi(q) = q^{-D}$. Then
\[
\bigcap_{\epsilon > 0}W_{\epsilon\psi}^\ext = \emptyset.
\]
\end{theorem}
\begin{proof}
Without loss of generality assume that $P_1,\ldots,P_n$ have integral coefficients. By contradiction, suppose there exists $\xx\in \bigcap_{\epsilon > 0}W_{\epsilon\psi}^\ext$. Fix $\epsilon > 0$ to be determined, and suppose that $\dist(\pp/q,\xx)\leq \epsilon\psi(q) = \epsilon q^{-D}$ for some $\pp/q\in\Q^d\butnot M$. Since the derivatives $P_1',\ldots,P_n'$ are bounded in a neighborhood of $\xx$, we have
\[
|P_i(\pp/q)| \lesssim \epsilon q^{-D} \all i = 1,\ldots,n.
\]
On the other hand, direct calculation shows that $P_i(\pp/q)$ is a rational number of denominator no greater than $q^D$. Thus if $\epsilon$ is sufficiently small, then
\[
P_i(\pp/q) = 0 \all i = 1,\ldots,n.
\]
But then $\pp/q\in M$, a contradiction.
\end{proof}
}

\ignore{

For the purposes of this paper, the open set condition is not enough to guarantee good behavior; see \sectionsymbol\ref{} for a detailed discussion. Instead we will need the following more stringent condition:

\begin{definition}
An IFS $(u_a)_{a\in E}$ satisfies the \emph{convex open set condition} if it satisfies the open set condition for some open set $W$ which is the finite union of convex open sets.
\end{definition}

An example of an IFS which satisfies the open set condition but not the convex open set condition is given in \sectionsymbol\ref{(see ignored material below)}. Nevertheless, the IFSs generating most well-known fractals such as the Cantor set and the von Koch curve satisfy the convex open set condition.

}

\ignore{
\begin{remark}
The following IFS satsifies the open set condition but not the convex open set condition:
\begin{align*}
u_1(z) &= \frac{1}{3}iz\\
u_2(z) &= \frac{1}{3}z + 1\\
u_3(z) &= \frac{1}{3}iz + 1.
\end{align*}
\end{remark}
[Add in proof\internal]
}

\section{Sketch of a proof of Corollary \ref{corollarycantor}; Preliminaries}
\label{sectionpreliminaries}

Before presenting the proofs of Theorem \ref{theoremdirichlet}, we sketch a short proof of Corollary \ref{corollarycantor} which uses the theory of continued fractions. This proof contains the basic idea of the more general proof, but it has the advantage of being more intuitive to someone familiar with the theory of continued fractions. After sketching the proof of Corollary \ref{corollarycantor}, we begin the preliminaries for the proof of Theorem \ref{theoremdirichlet}, using the sketch as motivation.

\subsection{Sketch of a proof of Corollary \ref{corollarycantor}}
We recall some elements of the theory of continued fractions (e.g. \cite{Khinchin_book}). As a matter of notation, given $a_1,\ldots \in\N$ we let
\[
[0;a_1,\ldots,a_n] = \cfrac{1}{a_1 + \cfrac{1}{\ddots + \cfrac{1}{a_n}}}
\]
and let $[0;a_1,\ldots] = \lim_{n\to\infty}[0;a_1,\ldots,a_n]$. Any $x\in(0,1)\butnot\Q$ can be represented uniquely as $x = [0;a_1,\ldots]$ for some sequence $(a_n)_1^\infty\in\N^\N$. The rationals $p_n/q_n = [0;a_1,\ldots,a_n]$ are called the \emph{convergents} of $x$. We have the following (see e.g. \cite[(30) and Theorem 19]{Khinchin_book}):
\begin{equation}
\label{khinchinimplications}
\left|x - \frac{p}{q}\right| < \frac{1}{2q^2} \;\;\Rightarrow\;\; \frac{p}{q} \text{ is a convergent of $x$} \;\;\Rightarrow\;\; \left|x - \frac{p}{q}\right| < \frac{1}{q^2}\cdot
\end{equation}
Now let $K$ denote the Cantor set, and suppose $x\in K\butnot\Q$. If all but finitely many convergents of $x$ lie in the Cantor set, then \eqref{khinchinimplications} shows that Corollary \ref{corollarycantor} cannot be true with $C = 1/2$; if infinitely many convergents of $x$ lie outside of the Cantor set, then \eqref{khinchinimplications} shows that the conclusion of Corollary \ref{corollarycantor} holds with $C = 1$. However, the question of whether infinitely many convergents of $x$ must lie outside the Cantor set is a difficult question, which we do not attempt to address here. Instead, we adopt the more modest approach of looking beyond the set of convergents; we are content to raise the value of $C$ if necessary.

\begin{proof}[Sketch of the proof of Corollary \ref{corollarycantor}]
(For a rigorous proof of Corollary \ref{corollarycantor}, see Section \ref{sectiondirichlet}.) Let $K$ denote the Cantor set, and suppose $x = [0;a_1,\ldots]\in K\butnot\Q$. Fix $n\in\N$, and for each $b\in\N$ let
\[
\frac{p_{n,b}}{q_{n,b}} = [0;a_1,\ldots,a_{n - 1},b].
\]
Then $p_{n,a_n}/q_{n,a_n} = p_n/q_n$, but otherwise $p_{n,b}/q_{n,b}$ is not a convergent of $x$. Now fix $N\in\N$, and consider the finite sequence $(p_{n,b}/q_{n,b})_{b = a_n}^{a_n + N}$. It can be shown that
\begin{itemize}
\item[(i)] $q_{n,b} \asymp_N q_n$ for $b = a_n,\ldots,a_n + N$, and that
\item[(ii)] the sequence $(p_{n,b}/q_{n,b})_{b = a_n}^{a_n + N}$ is roughly an arithmetic progression of increment $1/q_n^2$.
\end{itemize}
(The notion of a ``roughly arithmetic progression'' will be made precise in Definition \ref{definitionroughlyarithmetic}.) In particular, for each $b = a_n,\ldots,a_n + N$,
\[
\left|x - \frac{p_{n,b}}{q_{n,b}}\right| \lesssim \frac{|a_n - b| + 1}{q_n^2},
\]
and so
\[
\left|x - \frac{p_{n,b}}{q_{n,b}}\right| \lesssim_N \frac{1}{q_{n,b}^2}\cdot
\]
Let $C_N$ be the implied constant. If we assume for a contradiction that \eqref{extrinsic1dim} holds for only finitely many $p/q$, then there exists $n$ such that for all $b = a_n,\ldots,a_n + N$, we have $p_{n,b}/q_{n,b}\in K$. It follows that $K$ contains arbitrarily large roughly arithmetic progressions.

On the other hand, it is easily seen that $K$ does not contain any arithmetic progression of length five. (If such a progression existed, then the largest such progression would need to contain points from both $[0,1/3]$ and $[2/3,1]$, so its increment would need to be at least $1/3$.) A similar argument shows that there exists $N$ for which $K$ does not contain any roughly arithmetic progressions of length $N$ (cf. Proposition \ref{propositionroughlyarithmetic} below). This is a contradiction.
\end{proof}

\subsection{Good pairs of rational approximations}

In higher dimensions, we cannot use the theory of continued fractions, but we will still produce a sequence of roughly arithmetic progressions which consist of good rational approximations to the desired point. To see how this generalization will work, note that we can write (see e.g. \cite[Theorem 1]{Khinchin_book})
\[
\frac{p_{n,b}}{q_{n,b}} = \frac{p_{n - 2} + bp_{n - 1}}{q_{n - 2} + bq_{n - 1}}\cdot
\]
In particular, $\big((p_{n,b},q_{n,b})\big)_{b\in\N}$ is a true arithmetic progression in $\Z^2$, whose initial value $(p_{n - 2},q_{n - 2})$ and increment $(p_{n - 1},q_{n - 1})$ both represent good rational approximations to $x$. In higher dimensions, we will use the same principle, taking an arithmetic progression in $\Z^{d + 1}$ and projectivizing to get a roughly arithmetic progression in $\Q^d$.

The initial value and increment of our progression must both be good approximations, but they will not be chosen independently; they should be roughly ``on the same order of magnitude''. We make this rigorous in the following lemma:

\begin{lemma}
\label{lemmagoodpair}
Fix $\xx\in\R^d\butnot\Q^d$. Then for every $Q > 0$, there exists a pair $(\rr_0,\rr_\infty)\in(\Z^{d + 1})^2$ such that
\begin{itemize}
\item[(i)] $\rr_0,\rr_\infty$ are linearly independent;
\item[(ii)] If we write $\rr_i = (\pp_i,q_i)$ ($i = 0,\infty$), then $0\leq q_\infty\leq q_0$ and
\begin{equation}
\label{goodpair}
\|q_i\xx - \pp_i\| \leq \frac{1}{q_0^{1/d}};
\end{equation}
\end{itemize}
and such that $q_0\geq Q$.
\end{lemma}
We will call a pair $(\rr_0,\rr_\infty)$ satisfying (i) and (ii) a \emph{good pair for $\xx$}.
\begin{proof}[Proof of Lemma \ref{lemmagoodpair}]
Interpret $\xx$ as a column vector and let
\begin{align*}
T_\xx &:= \left[\begin{array}{cc}
I_d & -\xx\\
0 & 1
\end{array}\right]\\
g_t &:= \left[\begin{array}{cc}
e^{t/d}I_d & 0\\
0 & e^{-t}
\end{array}\right].
\end{align*}
Here $I_d$ denotes the $d\times d$ identity matrix.
\begin{claim}
There exists a sequence $t_k\tendsto k \infty$ such that for each $k\in\N$,
\[
\lambda_2(g_{t_k}\circ T_\xx(\Z^{d + 1})) \leq 1,
\]
where $\lambda_2$ is the second successive Minkowski minimum\footnote{See e.g. \cite[\6IV.1]{Schmidt3} for an exposition of Minkowski's theory of successive minima.} (with respect to the max norm).
\end{claim}
\begin{subproof}
By contradiction, suppose that there exists $t_0 > 0$ such that for all $t > t_0$ we have $\lambda_2(g_t\circ T_\xx(\Z^{d + 1})) > 1$, and let $U = (t_0,\infty)$.

Let $\Z^{d + 1}_\prim$ denote the set of primitive vectors of $\Z^{d + 1}$. For each $\rr\in\Z^{d + 1}_\prim$ let
\[
U_\rr = \{t\in U:\|g_t\circ T_\xx(\rr)\| < 1\}.
\]
We claim that the collection of sets $(U_\rr)_{\rr\in\Z^{d + 1}_\prim}$ is a disjoint open cover of $U$. Indeed, if $t\in U_{\rr_1}\cap U_{\rr_2}$ for some $\rr_1,\rr_2\in\Z^{d + 1}_\prim$ for which $U_{\rr_1}\neq U_{\rr_2}$, then we would have $\lambda_2(g_t\circ T_\xx(\Z^{d + 1})) < 1$, contradicting our hypothesis. On the other hand, for any $t\in U$, we have by Minkowski's second theorem
\[
\lambda_1(g_t\circ T_\xx(\Z^{d + 1})) \leq \frac{1}{\prod_{i = 2}^{d + 1} \lambda_i(g_t\circ T_\xx(\Z^{d + 1}))} \leq \frac{1}{\lambda_2(g_t\circ T_\xx(\Z^{d + 1}))^d} < 1,
\]
and so there exists $\rr\in\Z^{d + 1}_\prim$ such that $\|g_t\circ T_\xx(\rr)\| < 1$, i.e. $t\in U_\rr$.

Since $U$ is a connected set, it follows that $U = U_\rr$ for some $\rr\in\Z^{d + 1}_\prim$. Then $\|g_t\circ T_\xx(\rr)\|$ is bounded as $t$ tends to infinity. Thus $T_\xx(\rr)\in\R\ee_{d + 1}$, i.e. $\xx = \pp/q$ where $\rr = (\pp,q)$, contradicting that $\xx\notin\Q^d$.
\end{subproof}
Now fix $k\in\N$ and let $t = t_k$. Since $\lambda_2(g_t\circ T_\xx(\Z^{d + 1})) \leq 1$, there exist $\rr_0,\rr_\infty\in\Z^{d + 1}$ linearly independent such that
\[
\|g_t\circ T_\xx(\rr_i)\| \leq 1, \;\; i = 0,\infty.
\]
Writing this out in terms of coordinates yields
\[
\|e^{t/d}(\pp_i - q_i\xx)\| \leq 1 \text{ and } |e^{-t}q_i| \leq 1,
\]
or equivalently
\begin{equation}
\label{qixpi}
\|q_i\xx - \pp_i\| \leq e^{-t/d} \text{ and } |q_i| \leq e^t.
\end{equation}
By replacing $\rr_i$ by $-\rr_i$ if necessary, we may assume $q_i\geq 0$ for $i = 0,\infty$; by swapping $\rr_0$ and $\rr_\infty$ if necessary we may assume $q_\infty\leq q_0$. After these reductions, \eqref{qixpi} implies \eqref{goodpair}, which demonstrates that the pair $(\rr_0,\rr_\infty)$ is good for $\xx$. Finally, we observe that as $k\to\infty$ we have $q_0\to\infty$, because otherwise \eqref{qixpi} would imply that $\xx\in\Q^d$. Thus for any given $Q > 0$, we can construct a good pair for which $q_0\geq Q$.
\end{proof}

Suppose that $(\rr_0,\rr_\infty)$ is a good pair for $\xx$. We now create an arithmetic progression in $\Z^{d + 1}$ using $\rr_0$ as the initial value and $\rr_\infty$ as the increment. For each $i\in\Z$ let
\begin{equation}
\label{ridef}
\rr_i = \rr_0 + i\rr_\infty
\end{equation}
and write
\[
\rr_i = (\pp_i,q_i).
\]
Then $\{\rr_i : i\in\Z\}$ is an arithmetic progression in $\Z^{d + 1}$.

Our key claim is that \emph{every} rational $\pp_i/q_i$ represents a good approximation to $\xx$ in the sense of \eqref{extrinsic}, with $C$ depending only on $i$ and not on $\xx$.

\begin{claim}
\label{claimextrinsic2}
For $i\in\N$,
\begin{equation}
\label{extrinsic2}
\left\|\xx - \frac{\pp_i}{q_i}\right\| \leq \left(\frac{1 + i}{q_i}\right)^{1 + 1/d}\cdot
\end{equation}
\end{claim}
\begin{proof}
\begin{align*}
\|q_i\xx - \pp_i\|
&= \|(q_0\xx - \pp_0) + i(q_\infty\xx - \pp_\infty)\| \\
&\leq \frac{1 + i}{q_0^{1/d}} \by{\eqref{goodpair}}\\
q_i &= q_0 + iq_\infty \leq (1 + i)q_0
\end{align*}
and rearranging gives the desired result.
\end{proof}
We remark that for $N\in\N$ fixed, if we let $C_N = (1 + N)^{1 + 1/d}$, then \eqref{extrinsic2} implies \eqref{extrinsic} for $i = 0,\ldots,N$.

\subsection{Roughly arithmetic progressions}
We would now like to make rigorous in what sense the sequence $(\pp_i/q_i)_0^\infty$ described above is a roughly arithmetic progression. We begin with the following observation:

\begin{observation}
\label{observationline}
For each $i\in\Z$, $\pp_i/q_i$ is on the line spanned by $\pp_0/q_0$ and $\pp_\infty/q_\infty$.
\end{observation}
\begin{proof}
The line spanned by $\pp_0/q_0$ and $\pp_\infty/q_\infty$ is the projectivization of the two-dimensional subspace of $\R^{d + 1}$ spanned by $\rr_0$ and $\rr_\infty$.
\end{proof}

Thus, it will not be too restrictive for us to require roughly arithmetic progressions to be subsets of lines.

Let $L$ be a line in $\R^d$, and let $L_0$ be its linear part, i.e. $L_0 = L - \xx$ where $\xx\in L$ is any point. One way of defining arithmetic progressions on $L$ is to say that a sequence $(\xx_i)_0^N$ is an arithmetic progression if there exists a vector $\vv\in L_0$ (the \emph{increment}) such that for all $0\leq i < j \leq N$,
\begin{equation}
\label{lji}
\xx_j - \xx_i = (j - i)\vv.
\end{equation}
To define a roughly arithmetic progression, we will relax the condition \eqref{lji}. Specifically, we have the following:
\begin{definition}
\label{definitionroughlyarithmetic}
Fix $C > 0$. A sequence $(\xx_i)_0^N$ in $L$ is a \emph{$C$-roughly arithmetic progression} if there exists $\vv\in L_0\butnot\{\0\}$ such that for all $0\leq i < j \leq N$,
\begin{equation}
\label{ljiapprox}
\frac{1}{C}(j - i) \leq \frac{\xx_j - \xx_i}{\vv} \leq C(j - i).
\end{equation}
Here the expression $\frac{\xx_j - \xx_i}{\vv}$ denotes the unique value $c\in\R$ for which $\xx_j - \xx_i = c\vv$.
\end{definition}

In fact, the sequence $(\pp_i/q_i)_0^\infty$ is not a roughly arithmetic progression in this sense, since $\|\pp_j/q_j - \pp_i/q_i\| \to 0$ as $i,j\to\infty$. However, we will now prove that sufficiently long subsequences of the sequence $(\pp_i/q_i)_0^\infty$ are roughly arithmetic.

\begin{proposition}
\label{propositionroughlyarithmetic}
Fix $\pp_0/q_0,\pp_\infty/q_\infty\in\Q^d$, and for each $i\in\N$ let $\pp_i = \pp_0 + i\pp_\infty$ and $q_i = q_0 + i q_\infty$. Then for each $N\in\N$, $(\pp_i/q_i)_N^{2N}$ is a $2$-roughly arithmetic progression.
\end{proposition}
\begin{proof}
We observed above (Observation \ref{observationline}) that the sequence $(\pp_i/q_i)_N^{2N}$ is collinear. Fix $N\leq i < j \leq 2N$. Then
\begin{align*}
\frac{\pp_j}{q_j} - \frac{\pp_i}{q_i}
&= \frac{1}{q_i q_j} \big[q_i \pp_j - q_j \pp_i\big]\\
&= \frac{1}{q_i q_j} \big[(q_0 + i q_\infty)(\pp_0 + j \pp_\infty) - (q_0 + j q_\infty)(\pp_0 + i \pp_\infty)\big]\\
&= \frac{1}{q_i q_j} (j - i) \big[q_0 \pp_\infty - q_\infty \pp_0 \big].
\end{align*}
So let
\[
\vv = \frac{q_0 \pp_\infty - q_\infty \pp_0}{q_N q_{2N}},
\]
so that
\[
\frac{\frac{\pp_j}{q_j} - \frac{\pp_i}{q_i}}{\vv} = \frac{q_N q_{2N}}{q_i q_j}\cdot
\]
The proposition follows on noting that
\[
q_{2N} = q_0 + 2Nq_\infty \leq 2q_0 + 2Nq_\infty = 2q_N.
\]
\end{proof}

Proposition \ref{propositionroughlyarithmetic} allows us to prove a preliminary version of Theorem \ref{theoremdirichlet}:

\begin{corollary}
\label{corollarydirichlet}
Fix $d\in\N$ and a set $S\subset\R^d$. Suppose that for some $N$, $S$ contains no $2$-roughly arithmetic progression of length $N$. Then there exists $C > 0$ such that for all $\xx\in S\butnot\Q^d$, there exist infinitely many $\pp/q\in\Q^d\butnot S$ satisfying
\begin{repequation}{extrinsic}
\left\|\xx - \frac{\pp}{q}\right\| \leq \frac{C}{q^{1 + 1/d}}\cdot
\end{repequation}
\end{corollary}
In other words, if $S$ contains no $2$-roughly arithmetic progression of length $N$, then $S$ satisfies the conclusion of Theorem \ref{theoremdirichlet} with $C$ independent of $\xx$.
\begin{proof}
Let $C = C_N = (1 + N)^{1 + 1/d}$, and fix $\xx\in S\butnot\Q^d$.

Fix a pair $(\rr_0,\rr_\infty)$ which is good for $\xx$, and for each $i\in\N$, let $\rr_i = (\pp_i,q_i)$ be defined by \eqref{ridef}. By Proposition \ref{observationline}, the sequence $(\pp_i/q_i)_N^{2N}$ is a $2$-roughly arithmetic progression, so by hypothesis, this sequence contains a point which is not in $S$, say $\pp_i/q_i\notin S$. By Claim \ref{claimextrinsic2}, \eqref{extrinsic} is satisfied for $\pp_i/q_i$. To summarize, for each good pair $(\rr_0,\rr_\infty)$, there is a rational $\pp_i/q_i\in \Q^d\butnot S$ satisfying \eqref{extrinsic} with $q_i\geq q_0$.

By Lemma \ref{lemmagoodpair}, for each $Q > 0$ there is a good pair satisfying $q_0\geq Q$. Thus by the above argument, there exists $\pp/q\in\Q^d\butnot S$ satisfying \eqref{extrinsic} such that $q\geq Q$. Thus there are infinitely many rationals $\pp/q\in \Q^d\butnot S$ satisfying \eqref{extrinsic}.
\end{proof}

Corollary \ref{corollarydirichlet} says that in order to prove the extrinsic analogue for points in a fractal or manifold, it is enough to demonstrate a uniform bound on the length of a $2$-roughly arithmetic progression contained in that set. Intuitively, the reason for this should be that $S$ contains no line segment by assumption. (If $S$ did contain a line segment, then it would automatically contain arbitrarily long arithmetic progressions, which would in particular be $C$-roughly arithmetic for every $C\geq 1$.) So it will be useful to know that in certain cases, the limit of roughly arithmetic progressions is a line segment. To make this rigorous, we recall the definition of the \emph{Hausdorff metric} on the space of compact subsets of a metric space $X$.

\begin{definition}
Let $\KK^*(X)$ denote the set of nonempty compact subsets of $X$. The \emph{Hausdorff distance} between two sets $K_1,K_2\in\KK^*(X)$ is the number
\[
\dist_H(K_1,K_2) := \max\left\{\max_{x\in K_1}\dist(x,K_2) , \max_{x\in K_2}\dist(x,K_1)\right\}.
\]
\end{definition}
For background on the Hausdorff metric, see \cite[\64.F]{Kechris}. The topology induced on $\KK^*(X)$ by the Hausdorff metric is called the \emph{Vietoris topology} (cf. \cite[Exercise 4.21]{Kechris}).

We may now prove the following lemma:

\begin{lemma}
\label{lemmahausdorfflimit}
Fix $C > 0$. For each $N$, suppose that $K_N\subset\R$ is a $C$-roughly arithmetic progression of length $(N + 1)$, whose left and right endpoints are equal to $0$ and $1$, respectively. Then
\[
K_N \tendsto N [0,1]
\]
in the Vietoris topology.
\end{lemma}
\begin{proof}
Write $K_N = (x_{N,i})_{i = 0}^N$. Then \eqref{ljiapprox} reads:
\[
\frac{j - i}{C} \leq \frac{x_{N,j} - x_{N,i}}{v_N} \leq C(j - i).
\]
Plugging in $i = 0$, $j = N$ shows that $0 < v_N \leq C/N$. Then, plugging in $j = i + 1$ shows that $x_{N,i + 1} - x_{N,i} \leq C^2/N$ for all $i = 0,\ldots,N - 1$. It follows that $[0,1]\butnot K_N$ cannot contain any interval of length greater than $C^2/N$. In particular,
\[
\dist(x,K_N) \leq C^2/(2N) \all x\in [0,1].
\]
Since $K_N\subset[0,1]$, this implies that $\dist_H(K_N,[0,1])\leq C^2/(2N)$. Since $C^2/(2N) \to 0$ as $N\to\infty$, this completes the proof.
\end{proof}

\subsection{Iterated function systems}
We now recall the notion of an iterated function system (IFS); for a detailed exposition see \cite[\69]{Falconer_book}. We will only consider the case of a finite IFS generated by similarities.

\begin{definition}
\label{definitionIFS}
Fix $d\in\N$, and let $E$ be a finite set. An \emph{iterated function system} (\emph{IFS}) on $\R^d$ is a collection $(u_a)_{a\in E}$ of contracting similarities $u_a:\R^d\to\R^d$ satisfying the \emph{open set condition}: there exists an open set $W\subset\R^d$ such that the collection $(u_a(W))_{a\in E}$ is a disjoint collection of subsets of $W$ (see \cite{Hutchinson} for a thorough discussion). The \emph{limit set} of the IFS is the image of the \emph{coding map} $\pi:E^\N\to\R^d$ defined by
\[
\pi(\omega) = \lim_{n\to\infty} u_{\omega_1}\circ\cdots\circ u_{\omega_n}(0).
\]
\end{definition}
\begin{remark}
By intersecting with a ball centered at $\0$ of sufficiently large radius, we may without loss of generality assume that the open set $W$ is bounded.
\end{remark}

Let us now introduce some notation. Let $E^* = \bigcup_{n\geq 0} E^n$. For $\omega\in E^*$, let $|\omega|$ denote the length of $\omega$, and let
\[
u_\omega = u_{\omega_1}\circ\cdots\circ u_{\omega_{|\omega|}},
\]
with the convention that $u_\smallemptyset$ is the identity map.

Although we believe the following lemma is well-known to experts, we include its proof for completeness.

\begin{lemma}
\label{lemmaOSC}
Let $(u_a)_{a\in E}$ be an IFS, and let $\Lambda$ denote the limit set of $(u_a)_{a\in E}$. There exists a constant $M\in\N$ such that for every set $S\subset\Lambda$, there exists a collection $A\subset E^*$ of cardinality at most $M$ with the following properties:
\begin{itemize}
\item[(i)] $S\subset \bigcup_{\omega\in A} u_\omega(\Lambda)$.
\item[(ii)] For all $\omega\in A$,
\begin{equation}
\label{uomegaS}
\|u_\omega'\| \leq \diam(S),
\end{equation}
where $\|u_\omega'\|$ denotes the contraction ratio of the similarity $u_\omega$.
\end{itemize}
\end{lemma}
\begin{proof}
Let $\w A$ be the set of all words $\omega\in E^*$ which satisfy \eqref{uomegaS} but for which no proper initial segment satisfies \eqref{uomegaS}, and let $A = \{\omega\in \w A: u_\omega(\Lambda)\cap S \neq \emptyset\}$. Then (i) and (ii) are satisfied. To complete the proof, we must show that $\#(A)$ is bounded independent of $S$.

If $\emptyset\in A$, then $\#(A) = 1$. Thus we may assume $\emptyset\notin A$. Fix $\omega\in A$. The minimality of $\omega$ impiles that
\[
\|u_{\omega\given [1,|\omega| - 1]}'\| > \diam(S),
\]
where $\given$ denotes restriction. In particular, letting $\gamma = \min_{a\in E}\|u_a'\| > 0$, we have
\[
\|u_\omega'\| = \|u_{\omega_{|\omega|}}'\|\cdot\|u_{\omega\given [1,|\omega| - 1]}'\| > \gamma\diam(S).
\]
Moreover, if $\xx\in S$ then
\[
u_\omega(W)\subset B(\xx,\diam(S) + \diam(u_\omega(W))) \subset B(\xx,\diam(S)(1 + \diam(W))).
\]
Since no word in $A$ is an initial segment of another word in $A$, the open set condition implies that the collection $(u_\omega(W))_{\omega\in A}$ is disjoint. Letting $\lambda$ denote Lebesgue measure, we have
\begin{align*}
\diam(S)^d
\asymp \lambda(B(\xx,\diam(S)(1 + \diam(W))))
&\geq \sum_{\omega\in A}\lambda(u_\omega(W))\\
&= \sum_{\omega\in A}\|u_\omega'\|^d\lambda(W)
\asymp \sum_{\omega\in A}\diam(S)^d.
\end{align*}
Dividing both sides by $\diam(S)^d$, we see that $\#(A)$ is bounded from above independent of $S$.
\end{proof}

\section{Proof of Theorem \ref{theoremdirichlet} (Extrinsic analogue of Dirichlet's theorem)}
\label{sectiondirichlet}

\begin{proof}[Proof of Theorem \ref{theoremdirichlet}, case \text{(1)}]
Suppose that $S$ is the limit set of the IFS $(u_a)_{a\in E}$, and write $\Lambda = S$. By Corollary \ref{corollarydirichlet}, to complete the proof it suffices to show that there exists $N$ such that $\Lambda$ contains no $2$-roughly arithmetic progression of length $N$.

By contradiction, suppose that $\Lambda$ contains arbitrarily long $2$-roughly arithmetic progressions. For each $N$, let $P_N = (\xx_i)_0^N$ be a $2$-roughly arithmetic progression of length $(N + 1)$, and let $\gamma_N:\R\to\R^d$ be an affine transformation such that $\gamma_N(0) = \xx_0$ and $\gamma_N(1) = \xx_N$. Then $K_N := \gamma_N^{-1}(P_N)$ is also a $2$-roughly arithmetic progression; moreover, the left and right endpoints of $K_N$ are $0$ and $1$, respectively.

Let $M\in\N$ be as in Lemma \ref{lemmaOSC}. Then by Lemma \ref{lemmaOSC}, for each $N\in\N$ there is a collection $A_N\subset E^*$ of cardinality at most $M$ such that
\begin{equation}
\label{PNsubset}
P_N\subset \bigcup_{\omega\in A_N}u_\omega(\Lambda)
\end{equation}
and
\begin{equation}
\label{uprimegammaprime}
\|u_\omega'\| \leq \diam(P_N) = \|\gamma_N'\| \all \omega\in A_N.
\end{equation}
Enumerate the elements of $A_N$ by $\omega^{(N,1)},\ldots,\omega^{(N,M_N)}$ with $M_N\leq M$. For each $j = 1,\ldots,M$ let
\[
K_{N,j} = \left\{x\in K_N: \gamma_N(x)\in u_{\omega^{(N,j)}}(\Lambda)\right\}
\]
if $j\leq M_N$, and $K_{N,j} = \emptyset$ otherwise. By \eqref{PNsubset},
\[
K_N = \bigcup_{j = 1}^M K_{N,j}.
\]
By the compactness of $\KK^*([0,1])$ under the Vietoris topology \cite[Theorem 4.26]{Kechris}, there exists an increasing sequence $(N_k)_1^\infty$ such that for each $j = 1,\ldots,M$, the sequence $(K_{N_k,j})_{k = 1}^\infty$ converges to a set $K_{\infty,j} \in \KK^*([0,1])$. Since the finite union operation is continuous in the Vietoris topology \cite[Exercise 4.29(iv)]{Kechris}, by Lemma \ref{lemmahausdorfflimit} we have
\[
\bigcup_{j = 1}^M K_{\infty,j} = \lim_{N\to\infty}K_N = [0,1].
\]
Now by elementary topology, the union of nowhere dense sets is nowhere dense, and so there exists $j = 1,\ldots,M$ such that the set $K_{\infty,j}$ contains a nontrivial interval $[a,b]\subset[0,1]$. For each $N\in\N$, define $h_N:[0,1]\to\R^d$ by
\[
h_N = u_{\omega^{(N,j)}}^{-1}\circ \gamma_N.
\]
By the definition of $K_{N,j}$, we have $h_N(K_{N,j})\subset \Lambda$. On the other hand, by \eqref{uprimegammaprime} we have
\[
\|h_N'\| \geq 1.
\]
Since $h_N(0)$ and $h_N(1)$ are in the bounded set $\Lambda$, $\|h_N'\|$ is bounded independent of $N$. Let $(N_k)_1^\infty$ be an increasing sequence which is a subsequence of the previously chosen sequence and for which the sequence of affine functions $h_{N_k}$ converges locally uniformly to a non-constant affine function $h:[0,1]\to\R^d$. The map $(h,K)\mapsto h(K)$ is continuous from $(\text{locally uniform topology} \times \text{Vietoris topology})$ to the Vietoris topology \cite[(16.11)]{Simmons3}; thus
\[
h_{N_k}(K_{N_k,j}) \tendsto k h(K_{\infty,j}).
\]
But $h_{N_k}(K_{N_k,j})\subset \Lambda$ by construction, so $h(K_{\infty,j})\subset \Lambda$ by \cite[Exercise 4.29(ii)]{Kechris}. Since $K_{\infty,j}\supset[a,b]$, we have
\[
\Lambda\supset h([a,b]),
\]
i.e. $\Lambda$ contains a line segment, contradicting our hypothesis.
\end{proof}
\begin{proof}[Proof of Theorem \ref{theoremdirichlet}, case \text{(2)}]
Write $M = S$. The first step of the proof is to show that since $M$ does not contain a line segment, the cardinality of its intersection with any short enough line segment is bounded from above. Rigorously:

\begin{claim}
\label{claimrealanalytic}
For each $\xx\in M$, there exist a neighborhood $U$ of $\xx$ and an integer $N\in\N$ such that for every line $L$,
\[
\#(U\cap L)\leq N.
\]
\end{claim}
\begin{subproof}
By the implicit function theorem, if $U\subset\R^d$ is a sufficiently small neighborhood of $\xx$, then there exist real-analytic functions $f_1,\ldots,f_s:U\to\R$ such that $M\cap U = \bigcap_{i = 1}^s f_i^{-1}(0)$, where $s = d - \dim(M)$. By contradiction, for each $N\in\N$ large enough so that $B(\xx,1/N)\subset U$, choose a line $L_N$ so that
\[
\#(M\cap B(\xx,1/N)\cap L_N) > N.
\]
Parameterize $L_N$ by an affine transformation $\gamma_N:\R\to\R^d$ satisfying $\|\gamma_N'\| = 1$. Since the unit sphere $S^{d - 1}$ is compact, we may choose a sequence $(N_k)_1^\infty$ and a vector $\vv\in S^{d - 1}$ so that $\gamma_{N_k}' \tendsto k \vv$. Let $(a_N,b_N) = \gamma_N^{-1}(B(\xx,1/N))$.

Fix $i = 1,\ldots,s$ and $N\in\N$. Then $f_i\circ \gamma_N$ has at least $N$ zeros on $(a_N,b_N)$, since each point in $M\cap B(\xx,1/N)\cap L_N$ corresponds to a joint zero of $f_1,\ldots,f_s$ on $(a_N,b_N)$.
\begin{claim}
For each $j\leq N$, $(f_i\circ\gamma_n)^{(j)}$ has at least $(N - j)$ zeros on $(a_N,b_N)$.
\end{claim}
\begin{subproof}
Suppose the claim is true for $j < N$, and let $a_N < c_1 < \ldots < c_{N - j} < b_n$ be zeros of $(f_i\circ\gamma_n)^{(j)}$. By the mean value theorem, for each $k = 1,\ldots,N - (j + 1)$ there exists $c_k'\in (c_k,c_{k + 1})$ which is a zero of $(f_i\circ\gamma_n)^{(j + 1)}$. This completes the inductive step.
\end{subproof}
Now fix $j < N$, and let $c_{N,j}\in (a_N,b_N)$ be a zero of $(f_i\circ\gamma_N)^{(j)}$. We observe that by the chain rule,
\[
0 = (f_i\circ\gamma_N)^{(j)}(c_{N,j}) = f_i^{(j)}\circ\gamma_N(c_{N,j})[\gamma_N',\ldots,\gamma_N'].
\]
Here there are $j$ copies of $\gamma_N'$. Note that we have used the fact that $\gamma_N$ is affine to eliminate all terms involving a second order or higher derivative of $\gamma_N$, and to interpret $\gamma_N'$ as a vector rather than as a function whose output is a vector. Since $\gamma_N(c_{N,j})\in B(\xx,1/N)$, we have by Taylor's theorem
\begin{equation}
\label{bytaylorstheorem}
|f_i^{(j)}(\xx)[\gamma_N',\ldots,\gamma_N']| \leq (1/N)^{j + 1}\sup_{\yy\in B(\xx,1/N)}\|f_i^{(j + 1)}(\yy)\|.
\end{equation}
On the other hand, since $\gamma_{N_k}' \tendsto k \vv$ we have
\[
f_i^{(j)}(\xx)[\gamma_{N_k}',\ldots,\gamma_N'] \tendsto k f_i^{(j)}(\xx)[\vv,\ldots,\vv],
\]
which together with \eqref{bytaylorstheorem} implies that
\begin{equation}
\label{existslineanalyticversion}
f_i^{(j)}(\xx)[\vv,\ldots,\vv] = 0 \all i = 1,\ldots,s \all j\in\N.
\end{equation}
Since $f_1,\ldots,f_s$ are real-analytic, \eqref{existslineanalyticversion} implies that
\[
f_i(\xx + t\vv) = 0 \all i = 1,\ldots,s \all t\in\R\text{ sufficiently small}.
\]
Reinterpreting this statement in terms of the manifold $M$, we see that for sufficiently small $\varepsilon > 0$, the line segment $\xx + [-\varepsilon,\varepsilon]\vv$ is contained in $M$. This contradicts our hypothesis.
\end{subproof}
\begin{remark}
\label{remarkCinfty}
In the above proof, and for the remainder of the proof of Theorem \ref{theoremdirichlet}, the only step where we need $M$ to be real-analytic is the step where we use \eqref{existslineanalyticversion} to deduce the existence of a line segment contained in $M$. If $M$ is assumed to be $\CC^\infty$ and if we assume that \eqref{existslineanalyticversion} does not hold for any pair $(\xx,\vv)$, then the conclusion of Theorem \ref{theoremdirichlet} holds.
\end{remark}

We now claim that for any compact set $K\subset M$, the conclusion of Theorem \ref{theoremdirichlet} holds, with the constant $C_\xx$ depending only on $K$. Indeed, fix such a $K$, let $V\subset M$ be a neighborhood of $K$ which is relatively compact in $M$, and let $\w K = \cl V$. We use the compactness of $\w K$ to change the local principle of Claim \ref{claimrealanalytic} to a global one:

\begin{claim}
\label{claimcorollary}
There exists $N\in\N$ such that for every line $L$,
\[
\#(\w K\cap L) < N.
\]
\end{claim}
\begin{subproof}
For each $\xx\in M$, let $U_\xx$ and $N_\xx$ be as in Claim \ref{claimrealanalytic}. Then $(U_\xx)_{\xx\in \w K}$ is a cover of $\w K$; let $(U_{\xx_i})_{i = 1}^k$ be a finite subcover. The corollary holds with $N = \sum_{i = 1}^k N_{\xx_i} + 1$.
\end{subproof}

Thus $\w K$ contains no collinear $N$-tuples of distinct points, and in particular, $\w K$ contains no $2$-roughly arithmetic progressions of length $N$. So by Corollary \ref{corollarydirichlet}, there exists $C > 0$ such that for all $\xx\in K\butnot\Q^d$, there exist infinitely many $\pp/q\in\Q^d\butnot \w K$ satisfying \eqref{extrinsic}. Since $K$ is contained in the interior of $\w K$ relative to $M$, only finitely many of these rational points can satisfy $\pp/q\in M\butnot\w K$. Thus there infinitely many $\pp/q\in\Q^d\butnot M$ satisfying \eqref{extrinsic}.
\end{proof}

\ignore{

\section{Proof of Theorem \ref{theoremextrinsicfractals}}

\begin{theorem}
\label{theoremextrinsicfractals}
Let $J\subset\R^d$ be the limit set of an IFS, and assume that $J$ does not contain any line segment. Then there exists $C > 0$ such that for all $\xx\in J\butnot\Q^d$, there exist infinitely many $\pp/q\in\Q^d\butnot J$ satisfying
\begin{equation}
\label{extrinsic}
\left\|\xx - \frac{\pp}{q}\right\| \leq \frac{C}{q^{1 + 1/d}}\cdot
\end{equation}
\end{theorem}
We proceed by way of contradiction. Fix $N\in\N$, and suppose that the conclusion of Theorem \ref{theoremextrinsicfractals} does not hold for $C = (1 + N)^{1 + 1/d}$. Then there exists $\xx_N\in J\butnot\Q^d$ such that only finitely many rationals $\pp/q\in\Q^d\butnot J$ satisfy \eqref{extrinsic}. Choose $Q_N > 0$ larger than the denominators of each of these rationals, and apply Lemma \ref{lemmagoodpair} to get a good pair $(\rr_0^{(N)},\rr_\infty^{(N)})$ with $q_0^{(N)}\geq Q_N$. For $s\in\R$, let $\rr_s^{(N)}$ be defined as in \eqref{ridef}. Then by Claim \ref{claimextrinsic2}, $\rr_i^{(N)}$ satisfies \eqref{extrinsic} for $i = 0,\ldots,N$; thus
\begin{equation}
\label{negation}
\frac{\pp_i^{(N)}}{q_i^{(N)}}\in J \text{ for } i = 0,\ldots,N.
\end{equation}

As indicated in Section 2, the idea now is to show that as $N$ tends to infinity, the sequence $(\pp_i^{(N)}/q_i^{(N)})_{i = 0}^N$ looks more and more like an arithmetic progression. Moving these roughly arithmetic progressions to the large scale and using compactness to find a limit, the result is a conformal image of a line segment which is contained in $J$. This will contradict our hypothesis.

We make the notion of a roughly arithmetic progression rigorous in the following lemma:

\begin{lemma}
\label{lemmaroughlyarithmetic}
For all $s\geq 0$ we have
\[
\frac{\displaystyle \left\|\left(\frac{\del}{\del s}\right)^2\left[\frac{\pp_s^{(N)}}{q_s^{(N)}}\right]\right\|}{\left\|\displaystyle \left(\frac{\del}{\del s}\right)\left[\frac{\pp_s^{(N)}}{q_s^{(N)}}\right]\right\|}
\leq \frac{2}{1 + s}\cdot
\]
\end{lemma}
\begin{proof}
For convenience we omit the superscript of $N$.
\begin{align*}
\frac{\pp_s}{q_s} &= \frac{\pp_0 + s \pp_\infty}{q_0 + s q_\infty}\\
\left(\frac{\del}{\del s}\right)\left[\frac{\pp_s}{q_s}\right] &= \frac{\pp_\infty q_0 - q_\infty \pp_0}{(q_0 + s q_\infty)^2}\\
\frac{\displaystyle \left\|\left(\frac{\del}{\del s}\right)^2\left[\frac{\pp_s}{q_s}\right]\right\|}{\left\|\displaystyle \left(\frac{\del}{\del s}\right)\left[\frac{\pp_s}{q_s}\right]\right\|}
&= \frac{\displaystyle \left|\left(\frac{\del}{\del s}\right)\left[\frac{1}{(q_0 + sq_\infty)^2}\right]\right|}{\left|\displaystyle \left[\frac{1}{(q_0 + sq_\infty)^2}\right]\right|}\\
&= \left|\left(\frac{\del}{\del s}\right)\log\left[\frac{1}{(q_0 + sq_\infty)^2}\right]\right|\\
&= \left|-2 \left(\frac{\del}{\del s}\right)\log[q_0 + s q_\infty]\right|\\
&= 2\frac{q_\infty}{q_0 + s q_\infty}\\
&\leq 2\frac{q_\infty}{q_\infty + s q_\infty} \since{$q_\infty\leq q_0$}\\
&= \frac{2}{1 + s}\cdot
\end{align*}
\end{proof}
This lemma indicates that the function $s\mapsto \pp_s^{(N)}/q_s^{(N)}$ gets ``more linear'' as $s$ tends to infinity.

Now define $f_N:[0,1]\rightarrow\R$ by
\[
f_N(s) = N - s\lfloor\sqrt N\rfloor
\]
and $g_N:[0,1]\rightarrow\R^d$ by
\[
g_N(s) = \frac{\pp_{f_N(s)}^{(N)}}{q_{f_N(s)}^{(N)}}\cdot
\]
The following observation follows directly from (\ref{negation}):
\begin{observation}
\label{observationprogression}
Let
\[
K_N = \left\{\frac{i}{\lfloor\sqrt N\rfloor}: i = 0,\ldots,\sqrt N\right\}.
\]
Then $g_N(K_N)\subset J$.
\end{observation}
\begin{lemma}
\label{lemmaprogression}
\[
\sup_{[0,1]}\frac{\|g_N''\|}{\|g_N'\|} \lesssim \frac{1}{\sqrt N}\cdot
\]
\end{lemma}
\begin{proof}
Since $f_N$ is affine,
\begin{align*}
\frac{\|g_N''(s)\|}{\|g_N'(s)\|} &= \frac{\displaystyle \left\|\left(\frac{\del}{\del y}\right)^2\left[\frac{\pp_y^{(N)}}{q_y^{(N)}}\right]\right\|}{\left\|\displaystyle \left(\frac{\del}{\del y}\right)\left[\frac{\pp_y^{(N)}}{q_y^{(N)}}\right]\right\|}\Big|_{y = f_N(s)}\cdot |f_N'(s)|\\
&\leq \frac{2}{1 + f_N(s)}|f_N'(s)| \by{Lemma \ref{lemmaroughlyarithmetic}}\\
&\leq \frac{2}{1 + N - \lfloor\sqrt N\rfloor}\lfloor\sqrt N\rfloor \asymp \frac{1}{N}\sqrt N = \frac{1}{\sqrt N}\cdot
\end{align*}
\end{proof}
Observation \ref{observationprogression} and Lemma \ref{lemmaprogression} can be taken together as formalizing the idea that $J$ contains longer and longer sequences which get closer and closer to being arithmetic progressions.

Let $L_N = g_N([0,1])$. By Observation \ref{observationline}, $L_N$ is a line segment. Let $W$ be the set from the convex open set condition. By intersecting with $B(\0,R)$ for sufficiently large $R$, we may assume that $W$ has finite diameter. Let
\[
A_N = \bigcup_{n\in\N}\left\{\omega\in E^n: u_\omega(\cl W)\cap L_N \neq \emptyset, \; \|u_\omega'\| \leq \frac{\diam(L_N)}{\diam(W)}, \text{ and } \|u_{\omega_1^{n - 1}}'\| > \frac{\diam(L_N)}{\diam(W)} \right\}.
\]
Then for each $\omega\in A_N$, we have $\|u_\omega'\| > \gamma\diam(L_N)/\diam(W)$, where $\gamma = \min_{a\in E}\|u_a'\|$. Moreover,
\[
u_\omega(W)\subset B(L_N,\diam(L_N)).
\]
Since any two strings in $A_N$ are incomparable, the collection $(u_\omega(W))_{\omega\in A_N}$ is disjoint. Thus
\begin{align*}
\diam(L_N)^d
\asymp \lambda(B(L_N,\diam(L_N)))
&\geq \sum_{\omega\in A_N}\lambda(u_\omega(W))\\
&= \sum_{\omega\in A_N}\|u_\omega'\|^d\lambda(W)
\asymp \sum_{\omega\in A_N}\diam(L_N)^d.
\end{align*}
Dividing both sides by $\diam(L_N)^d$, we see that $\#(A_N)$ is bounded from above, say by $M$. Enumerate the elements of $A_N$ by $\omega^{(N,1)},\ldots,\omega^{(N,M_N)}$ with $M_N\leq M$. For each $j = 1,\ldots,M$ let
\[
K_{N,j} = \left\{x\in K_N: g_N(x)\in u_{\omega^{(N,j)}}(J)\right\}
\]
if $j\leq M_N$, and $K_{N,j} = \emptyset$ otherwise. By Observation \ref{observationprogression},
\[
K_N = \bigcup_{j = 1}^M K_{N,j}.
\]
Let $(N_k)_1^\infty$ be an increasing sequence such that for each $j = 1,\ldots,M$, the sequence $(K_{N_k,j})_{k = 1}^\infty$ converges in the Hausdorff metric to a set $K_{\infty,j}$; this is possible since the set of closed subsets of $[0,1]$ with the Hausdorff metric is a compact metric space. Since the finite union operation is continuous in the Hausdorff metric, we have
\[
\bigcup_{j = 1}^M K_{\infty,j} = \lim_{N\to\infty}K_N = [0,1].
\]
Now by elementary topology, the union of nowhere dense sets is nowhere dense, and so there exists $j = 1,\ldots,M$ such that the set $K_{\infty,j}$ contains a nontrivial interval $[a,b]\subset[0,1]$. For each $N\in\N$, define $h_N:[0,1]\to\R^d$ by
\[
h_N = u_{\omega^{(N,j)}}^{-1}\circ g_N.
\]
By the definition of $K_{N,j}$, we have $h_N(K_{N,j})\subset J$. On the other hand, by Lemma \ref{lemmaprogression}, we have
\[
\sup_{[0,1]}\frac{\|h_N''\|}{\|h_N'\|} \lesssim \frac{1}{\sqrt N}\cdot
\]
On the other hand,
\[
\|h_N'\| \asymp \|(u_{\omega^{(N,j)}}^{-1})'\|\diam(L_N) \asymp 1\selfnote{more details needed}
\]
and thus
\[
\sup_{[0,1]}\|h_N''\| \lesssim \frac{1}{\sqrt N}\cdot
\]
Let $(N_k)_1^\infty$ be an increasing sequence which is a subsequence of the previously chosen sequence and for which $h_{N_k}$ converges in the $\CC^2$ topology to a function $h:[0,1]\to\R^d$. Then for all $x\in[0,1]$
\begin{align*}
\|h'(x)\| = \lim_{N\to\infty}\|h_N'(x)\| &\asymp 1\\
\|h''(x)\| = \lim_{N\to\infty}\|h_N''(x)\| &=_\pt 0,
\end{align*}
which demonstrates that $h$ is a non-constant affine map. On the other hand, the map $(h,K)\mapsto h(K)$ is continuous from $(\CC^2\text{ topology},\text{Hausdorff topology})$ to the Hausdorff topology; thus
\[
h_{N_k}(K_{N_k,j}) \tendsto k h(K_{\infty,j}).
\]
But $h_{N_k}(K_{N_k,j})\subset J$ by construction, so $h(K_{\infty,j})\subset J$. Since $K_{\infty,j}\supset[a,b]$, we have
\[
J\subset h([a,b]),
\]
i.e. $J$ contains a line segment, contradicting our hypothesis.

}

\ignore{

\section{Completion of the proof in the case where the IFS satisfies SSC}
In this case, for each $N\in\N$ let $\omega\in E^n$ be the longest word so that $g_N([0,1])\subset u_\omega(\w{W})$. Then
\[
\|u_\omega'\|_W \asymp \diam(g_N([0,1]))
\]
and thus
\[
\diam(u_\omega^{-1}\circ g_N([0,1])) \asymp 1.
\]
On the other hand, for each $i = 0,\ldots,\lfloor\sqrt N\rfloor$ we have
\[
u_\omega^{-1}\circ g_N\left(\frac{i}{\lfloor\sqrt N\rfloor}\right)\in J
\]
but by the bounded distortion property, we have
\[
\dist\left(u_\omega^{-1}\circ g_N\left(\frac{i}{\lfloor\sqrt N\rfloor}\right),u_\omega^{-1}\circ g_N\left(\frac{i + 1}{\lfloor\sqrt N\rfloor}\right)\right) \asymp \frac{1}{\sqrt N}
\]
and thus if $N$ is sufficiently large, then $u_\omega^{-1}\circ g_N([0,1])\subset u_a(\cl{W})$ for some $a\in E$. This is a contradiction.

}

\section{Proof of Proposition \ref{propositionvonkoch} (The von Koch curve does not contain a line segment)}
\label{sectionvonkoch}

Recall (cf. \cite[\63.3(2)]{Hutchinson}) that the von Koch snowflake curve is the limit set of the IFS on $\R^2 \equiv \C$ generated by the similarities
\begin{align*}
u_1(z) &= \frac 13 z\\
u_2(z) &= \frac 13 e^{\pi i/3} z + \frac 13\\
u_3(z) &= \frac 13 e^{-\pi i/3} z + \frac 13 + \frac 13 e^{2\pi i/3}\\
u_4(z) &= \frac 13 z + \frac 23.
\end{align*}
This IFS satisfies the open set condition with respect to the equilateral triangle $W$ whose vertices are $0$, $1$, and $e^{\pi i/3}$ (cf. Figure \ref{figurevonkoch}). Denote the von Koch curve by $\Lambda$.

{\bf Convention.} In this proof, line segments are not considered to contain their endpoints.

By contradiction, suppose that the von Koch curve contains a line segment $L\subset \Lambda$. Without loss of generality, suppose that
\begin{itemize}
\item[(i)] The number of endpoints of $L$ contained in $\bigcup_a \del(u_a(W))$ is maximal among line segments contained in $\Lambda$.
\item[(ii)] The length of $L$ is maximal given (i).
\end{itemize}
We observe that $L$ cannot be contained in $u_a(W)$ for any $a\in E$; otherwise $u_a^{-1}(L)$ would be a line segment longer than $L$ but also satisfying (i). Since $L$ is connected, it follows that $L\butnot \bigcup_a u_a(W)\neq \emptyset$. Fix $x\in L\butnot \bigcup_a u_a(W) \subset L\cap \bigcup_a \del(u_a(W))$. Then $L\butnot\{x\}$ is the union of two line segments $L^1$ and $L^2$. Let $n$ denote the number of endpoints of $L$ contained in $\bigcup_a \del(u_a(W))$; we claim that $n = 2$. Indeed, if not, then either $L^1$ or $L^2$ has $(n + 1)$ endpoints contained in $\bigcup_a \del(u_a(W))$. (If $n = 0$, both $L^1$ and $L^2$ have this property; if $n = 1$, only one of them does.)

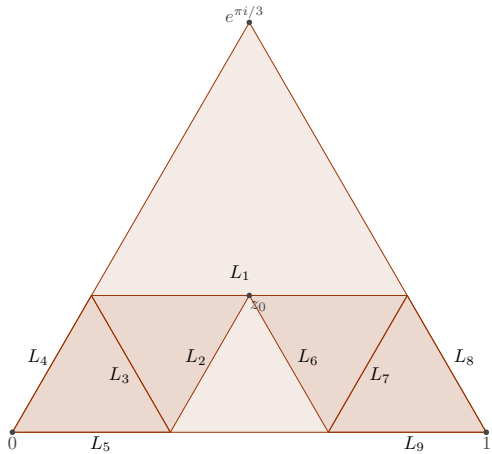
\begin{figure}
\scalebox{.7}{
\definecolor{xdxdff}{rgb}{0.49,0.49,1}
\definecolor{uuuuuu}{rgb}{0.27,0.27,0.27}
\definecolor{zzttqq}{rgb}{0.6,0.2,0}
\definecolor{qqqqff}{rgb}{0,0,1}
\begin{tikzpicture}[line cap=round,line join=round,>=triangle 45,x=1.0cm,y=1.0cm]
\clip(-4.25,-2.2) rectangle (4.95,6.5);
\fill[color=zzttqq,fill=zzttqq,fill opacity=0.1] (-4.16,-1.6) -- (4.84,-1.6) -- (0.34,6.19) -- cycle;
\fill[color=zzttqq,fill=zzttqq,fill opacity=0.1] (-4.16,-1.6) -- (-1.16,-1.6) -- (-2.66,1) -- cycle;
\fill[color=zzttqq,fill=zzttqq,fill opacity=0.1] (-2.66,1) -- (-1.16,-1.6) -- (0.34,1) -- cycle;
\fill[color=zzttqq,fill=zzttqq,fill opacity=0.1] (1.84,-1.6) -- (4.84,-1.6) -- (3.34,1) -- cycle;
\fill[color=zzttqq,fill=zzttqq,fill opacity=0.1] (0.34,1) -- (1.84,-1.6) -- (3.34,1) -- cycle;
\draw [color=zzttqq] (-4.16,-1.6)-- (4.84,-1.6);
\draw [color=zzttqq] (4.84,-1.6)-- (0.34,6.19);
\draw [color=zzttqq] (0.34,6.19)-- (-4.16,-1.6);
\draw [color=zzttqq] (-4.16,-1.6)-- (-1.16,-1.6);
\draw [color=zzttqq] (-1.16,-1.6)-- (-2.66,1);
\draw [color=zzttqq] (-2.66,1)-- (-4.16,-1.6);
\draw [color=zzttqq] (-2.66,1)-- (-1.16,-1.6);
\draw [color=zzttqq] (-1.16,-1.6)-- (0.34,1);
\draw [color=zzttqq] (0.34,1)-- (-2.66,1);
\draw [color=zzttqq] (1.84,-1.6)-- (4.84,-1.6);
\draw [color=zzttqq] (4.84,-1.6)-- (3.34,1);
\draw [color=zzttqq] (3.34,1)-- (1.84,-1.6);
\draw [color=zzttqq] (0.34,1)-- (1.84,-1.6);
\draw [color=zzttqq] (1.84,-1.6)-- (3.34,1);
\draw [color=zzttqq] (3.34,1)-- (0.34,1);
\draw (-0.16,1.7) node[anchor=north west] {$L_1$};
\draw (-1,0.1) node[anchor=north west] {$L_2$};
\draw (-2.45,-0.24) node[anchor=north west] {$L_3$};
\draw (-4,0.1) node[anchor=north west] {$L_4$};
\draw (-2.81,-1.55) node[anchor=north west] {$L_5$};
\draw (1.13,0.1) node[anchor=north west] {$L_6$};
\draw (2.5,-0.24) node[anchor=north west] {$L_7$};
\draw (4.1,0.1) node[anchor=north west] {$L_8$};
\draw (3.15,-1.55) node[anchor=north west] {$L_9$};
\fill [color=uuuuuu] (-4.16,-1.6) circle (1.5pt);
\draw[color=uuuuuu] (-4.16,-1.8) node {$0$};
\fill [color=uuuuuu] (4.84,-1.6) circle (1.5pt);
\draw[color=uuuuuu] (4.85,-1.8) node {$1$};
\fill [color=uuuuuu] (0.34,6.19) circle (1.5pt);
\draw[color=uuuuuu] (0.25,6.35) node {$e^{\pi i/3}$};
\fill [color=uuuuuu] (0.34,1) circle (1.5pt);
\draw[color=uuuuuu] (0.51,0.8) node {$z_0$};
\end{tikzpicture}
}
\caption{The open set and its first-level iterates for the von Koch snowflake curve.}
\label{figurevonkoch}
\end{figure}

To summarize: both endpoints $x_1,x_2$ of the line segment $L$ described by conditions (i) and (ii) are contained in $\bigcup_a \del(u_a(W))$; moreover, $L$ is not contained in $u_a(W)$ for any $a$. Let $L_1,\ldots,L_9$ be as in Figure \ref{figurevonkoch}, so that $\bigcup_a \del(u_a(W)) = \bigcup_1^9 L_i$. We now consider separately:
\begin{itemize}
\item[Case 1:] $x_1,x_2$ are contained in the same line segment $L_i$ for some $i$. In this case, we observe that since the intersection of $\Lambda$ with the $x$-axis is precisely the Cantor set, the intersection of $\Lambda$ with these line segments will be the union of finitely many images of the Cantor set under similarities. Since such a union cannot contain a line segment, this is a contradiction.
\item[Case 2:] $x_1,x_2$ lie on opposite sides of the line $L_{10} = \{\Re[z] = 1/2\}$. In this case, we observe that $\Lambda\cap L_{10}$ is a singleton $\{z_0\}$, where $z_0 = \frac 13 + \frac 13 e^{2\pi i/3}$ as in Figure \ref{figurevonkoch}. Since $L$ is connected, we must have $z_0\in L$. But since no point in $\Lambda$ has imaginary part greater than the imaginary part of $z_0$, the line $L$ must be horizontal. Thus $L\subset L_1$, and we are reduced to the first case.
\item[Case 3:] $x_1,x_2$ are contained in different line segments, and lie on the same side of $L_{10}$. In this case, without loss of generality we may assume that $x_1,x_2$ lie on the left hand side of $L_{10}$. Now if $\{x_1,x_2\}\subset L_1\cup L_2\cup L_3$ or $\{x_1,x_2\}\subset L_3\cup L_4\cup L_5$, then we would have $L\subset u_1(W)$ or $L\subset u_2(W)$, respectively. Either is a contradiction, so $\{x_1,x_2\}\nsubset L_1\cup L_2\cup L_3$ and $\{x_1,x_2\}\nsubset L_3\cup L_4\cup L_5$. It follows (after possibly swapping $x_1$ and $x_2$) that $x_1\in L_1\cap L_2\butnot L_3$ and $x_2\in L_4\cup L_5\butnot L_3$. In particular, we see that $L$ can be written as a union $L = L^1\cup\{y\}\cup L^2$, where $L^a\subset u_a(W)$ and $y\in L_3$. Now for one of $a = 1,2$, the length of $L^a$ is at least half of the length of $L$. But then $u_a^{-1}(L^a)$ is longer than $L$, contradicting (ii).
\end{itemize}

\ignore{

\section{Other results}

\begin{theorem}
\label{theoremkhinchinextrinsic}
Let $M$ and $\psi$ be as in Theorem \ref{theoremkhinchinambient}. If \eqref{khinchin} diverges, then $W_\psi^\ext$ has full $\lambda_M$-measure.
\end{theorem}
\begin{proof}
We begin by recalling the notation of \cite{FKMS}. For each $n,m\in\N$ let
\[
[n,m] = \binom{n + m}{m} = \binom{m + n}{n}.
\]
For each $k,d\in\N$, let $n_{k,d}\in\N$ be maximal such that
\begin{equation}
\label{mnkddef}
d = k + [k - 1,2] + \ldots + [k - 1,n_{k,d}] + m_{k,d}
\end{equation}
for some $m_{k,d}\geq 0$, and let $m_{k,d}$ be the unique integer satisfying \eqref{mnkddef}. Let
\[
N_{k,d} = k + 2[k - 1,2] + \ldots + n_{k,d}[k - 1,n_{k,d}] + (n_{k,d} + 1)m_{k,d},
\]
We observe that $N_{k,d} > d$ whenever $k < d$.
\begin{theorem}[{\cite[Theorem 5.5]{FKMS}}]
\label{theoremkhinchinintrinsic}
Fix $d\geq 2$, and let $M\subset\R^d$ be a nondegenerate real-analytic submanifold of dimension $k$. Then the set
\[
\VWA_M := \left\{\xx\in M : \exists \epsilon > 0 \;\; \exists^\infty \pp/q\in\Q^d\cap M \;\; \left\|\xx - \frac{\pp}{q}\right\| < \frac{1}{q^{(d + 1)/N_{k,d} + \epsilon}}\right\}
\]
has zero $\lambda_M$-measure.
\end{theorem}
In the notation of Section \ref{sectionotherresults},
\[
\VWA_M = \bigcup_{c > \frac{d + 1}{N_{k,d}}} W_{\psi_c}^\int,
\]
where $\psi_c(q) = q^{-c}$.

Now let $\psi$ be an approximation function, and suppose that \eqref{khinchin} diverges. If $\psi\leq K\psi_c$ for some $c > \frac{d + 1}{N_{k,d}}$ and $K > 0$, then $W_\psi^\int \subset \VWA_M$ has zero $\lambda_M$-measure by Theorem \ref{theoremkhinchinintrinsic}, and $W_\psi$ has full $\lambda_M$-measure by Theorem \ref{theoremkhinchinambient}.

On the other hand, suppose that $\psi \nleq K\psi_c$ for all $c > \frac{d + 1}{N_{k,d}}$ and $K > 0$. In particular, fix $\frac{d + 1}{N_{k,d}} < c_1 < c_2 < \frac{d + 1}{d}$ (this is possible since $N_{k,d} > d$). Then we have
\[
\psi(Q) > \psi_{c_1}(Q)
\]
for infinitely many $Q\in\N$. Fix such a $Q$. Since $\psi$ is decreasing by assumption, we have $\psi(q) > \psi_{c_1}(Q)$ for all $q < Q$. In particular,
\[
\sum_{q = 1}^\infty \min(\psi,\psi_{c_1})(q) \geq \sum_{q = 1}^Q \min(\psi_{c_1}(Q),\psi_{c_1}(q)) = \sum_{q = 1}^Q \psi_{c_1}(Q) = Q \psi_{c_1}(Q).
\]

\end{proof}

\begin{theorem}
\label{theoremjarnikextrinsic}
Let $M\subset\R^d$ be the zero set of the polynomials $P_1,\ldots,P_n:\R^d\to\R$, and assume that $P_1,\ldots,P_n$ have rational coefficients. Let $D$ be the maximum of the degrees of $P_1,\ldots,P_n$, and let $\psi(q) = q^{-D}$. Then
\[
\bigcap_{\epsilon > 0}W_{\epsilon\psi}^\ext = \emptyset.
\]
\end{theorem}
\begin{proof}
Without loss of generality assume that $P_1,\ldots,P_n$ have integral coefficients. By contradiction, suppose there exists $\xx\in \bigcap_{\epsilon > 0}W_{\epsilon\psi}^\ext$. Fix $\epsilon > 0$ to be determined, and suppose that $\dist(\pp/q,\xx)\leq \epsilon\psi(q) = \epsilon q^{-D}$ for some $\pp/q\in\Q^d\butnot M$. Since the derivatives $P_1',\ldots,P_n'$ are bounded in a neighborhood of $\xx$, we have
\[
|P_i(\pp/q)| \lesssim \epsilon q^{-D} \all i = 1,\ldots,n.
\]
On the other hand, direct calculation shows that $P_i(\pp/q)$ is a rational number of denominator no greater than $q^D$. Thus if $\epsilon$ is sufficiently small, then
\[
P_i(\pp/q) = 0 \all i = 1,\ldots,n.
\]
But then $\pp/q\in M$, a contradiction.
\end{proof}
}

\section{Metrical extrinsic approximation}
\label{sectionotherresults}

Theorem \ref{theoremdirichlet} gives an analogue of Dirichlet's theorem in the setting of extrinsic approximation. It is reasonable to ask whether analogues of the other classical theorems of Diophantine approximation, namely the Jarn\'ik--Schmidt, Khinchin, and Jarn\'ik--Besicovitch theorems (see e.g. \cite[Theorem III.2A]{Schmidt3} and \cite[Theorems 1.10 and 5.2]{Bugeaud}), also hold. In the case of the Jarn\'ik--Schmidt theorem, an extrinsic version can be deduced immediately from the ambient version, and in the case of Khinchin's theorem, an extrinsic version can be deduced from the ambient version together with a statement regarding intrinsic approximation which was proven in \cite{FKMS}. Finally, the Jarn\'ik--Besicovitch theorem is more subtle, and does not admit an extrinsic analogue with the same level of generality. We comment on this phenomenon in \6\ref{Remarks on the Jarnik--Besicovitch theorem} below.

\subsection{An analogue of the Jarn\'ik--Schimidt theorem}
The analogue of the Jarn\'ik--Schmidt theorem for ambient approximation on fractals and manifolds is the following:

\begin{theorem}[{\cite[Theorem 1.1]{BFKRW}} (cf. {\cite[Proposition 3.1]{BFS1}}) for fractals, {\cite[Theorem 1]{Beresnevich_BA}} for manifolds]
\label{theoremjarnikschmidtambient}
Fix $d\in\N$, and let $S\subset\R^d$ be either
\begin{itemize}
\item[(1)] the limit set of an iterated function system, or
\item[(2)] a real-analytic manifold.
\end{itemize}
Assume that $S$ is not contained in any proper affine subspace of $\R^d$. If
\[
\BA_d = \{\xx\in \R^d : \exists C > 0 \text{ for which \eqref{extrinsic} does not hold for any $\pp/q\in\Q^d$}\},
\]
then $\BA_d\cap S$ has full Hausdorff dimension in $S$.
\end{theorem}
The set $\BA_d$ is called the set of \emph{badly approximable} vectors.

We now claim that Theorem \ref{theoremjarnikschmidtambient} also implies an extrinsic analogue of the Jarn\'ik--Schmidt theorem. Indeed, the extrinsic analogue of $\BA_d\cap S$ is the set
\[
\BA^\ext := \{\xx\in S : \exists C > 0 \text{ for which \eqref{extrinsic} does not hold for any $\pp/q\in\Q^d\butnot S$}\},
\]
and it is a superset of $\BA_d\cap S$. So since $\BA_d\cap S$ has full Hausdorff dimension in $S$, so does $\BA^\ext$. This statement is what we refer to as the extrinsic analogue of the Jarn\'ik--Schmidt theorem.

\begin{remark}
The extrinsic analogue of the Jarn\'ik--Schmidt theorem can be viewed as demonstrating the optimality of Theorem \ref{theoremdirichlet}. Indeed, it demonstrates that for any function $\psi:\N\to(0,\infty)$ which decays faster than $q\mapsto q^{-(d + 1)/d}$, the statement which results from replacing the right hand side of \eqref{extrinsic} by $C\psi(q)$ in Theorem \ref{theoremdirichlet} cannot be true. See \cite{FishmanSimmons5} for a detailed discussion of such considerations.
\end{remark}

\subsection{An analogue of Khinchin's theorem}
When considering analogues of Khinchin's theorem, we consider only the case of manifolds. The case of fractals is more difficult, since the ambient analogue is not known; moreover, even if it were known, not enough is known about the intrinsic approximation theory of fractals to deduce an extrinsic version from a hypothetical ambient version.

The analogue of Khinchin's theorem for ambient approximation on manifolds is the following:

\begin{theorem}[{\cite[Theorem 2.3]{Beresnevich_Khinchin}}]
\label{theoremkhinchinambient}
Fix $d\geq 2$, and let $M$ be a real-analytic submanifold of $\R^d$ which is not contained in any proper affine subspace of $\R^d$. Let $\lambda_M$ denote Lebesgue measure on $M$. Then for any decreasing function $\psi:\N\to\R^+$, if the series
\begin{equation}
\label{khinchin}
\sum_{q\in\N}\psi(q)^d
\end{equation}
diverges,\footnote{One may also ask about the converse direction, namely whether the convergence of \eqref{khinchin} implies that $W_\psi$ has zero $\lambda_M$-measure. This is known in some cases; we refer to \cite{Beresnevich_Khinchin} for details.} then the set
\[
W_\psi := \left\{\xx\in M : \exists^\infty \pp/q\in\Q^d \;\; \left\|\xx - \frac{\pp}{q}\right\| < \frac{\psi(q)}{q}\right\}
\]
has full $\lambda_M$-measure.
\end{theorem}

The set $W_\psi$ is called the the set of \emph{$\psi$-approximable} points. Analogously, we define the set of \emph{$\psi$-intrinsically approximable} and \emph{$\psi$-extrinsically approximable} points
\begin{align*}
W_\psi^\int &:= \left\{\xx\in M : \exists^\infty \pp/q\in\Q^d\cap M \;\; \left\|\xx - \frac{\pp}{q}\right\| < \frac{\psi(q)}{q}\right\}\\
W_\psi^\ext &:= \left\{\xx\in M : \exists^\infty \pp/q\in\Q^d\butnot M \;\; \left\|\xx - \frac{\pp}{q}\right\| < \frac{\psi(q)}{q}\right\}.
\end{align*}
Obviously, $W_\psi = W_\psi^\int\cup W_\psi^\ext$. In particular, if $W_\psi$ has full $\lambda_M$-measure but $W_\psi^\int$ has zero $\lambda_M$-measure, then $W_\psi^\ext$ has full $\lambda_M$-measure. On the other hand, we have the following:

\begin{theorem}[Corollary of {\cite[Theorem 5.5]{FKMS}}]
\label{theoremkhinchinintrinsic}
Let $M$ and $\psi$ be as in Theorem \ref{theoremkhinchinambient}, and suppose that there exist $C,\epsilon > 0$ such that
\begin{equation}
\label{psiqbound}
\psi(q) \leq \frac{C}{q^\epsilon} \all q\in\N.
\end{equation}
Then  the set $W_\psi^\int$ has zero $\lambda_M$-measure.
\end{theorem}

Combining Theorems \ref{theoremkhinchinambient} and \ref{theoremkhinchinintrinsic}, we have the following:

\begin{theorem}
\label{theoremkhinchinextrinsic}
Let $M$ and $\psi$ be as in Theorem \ref{theoremkhinchinambient}. Suppose that \eqref{khinchin} diverges, and also that there exist $C,\epsilon > 0$ such that \eqref{psiqbound} holds. Then $W_\psi^\ext$ has full $\lambda_M$-measure.
\end{theorem}

\begin{remark}
In this context, the condition \eqref{psiqbound} is quite reasonable. Indeed, by Theorem \ref{theoremdirichlet}, we have
\begin{equation}
\label{dirichletM}
M = \bigcup_{C > 0}W_{C\psi_{1/d}}^\ext,
\end{equation}
where $\psi_c(q) := q^{-c}$. Intuitively, this means that the question about the Lebesgue measure of $W_\psi^\ext$ mostly makes sense if $\psi$ decays faster than $\psi_{1/d}$; if $\psi$ decays more slowly than $\psi_{1/d}$, then \eqref{dirichletM} implies that $W_\psi^\ext = M$, and so clearly $W_\psi^\ext$ has full measure in this case.
\end{remark}

\subsection{Remarks on the Jarn\'ik--Besicovitch theorem}
\label{Remarks on the Jarnik--Besicovitch theorem}
One may ask whether the above techniques can be used to prove an extrinsic analogue of the Jarn\'ik--Besicovitch theorem. Again, we consider only the case of manifolds. An analogue of the Jarn\'ik--Besicovitch theorem for ambient approximation on manifolds is proven in \cite[Theorem 2.5]{Beresnevich_Khinchin}; we omit the statement for conciseness, although we remark that it only applies to functions $\psi$ which decay more slowly than a fixed function $\psi_0$. However, it seems that not enough is known about intrinsic approximation on manifolds to deduce an extrinsic corollary.

To get an idea of what an extrinsic analogue of the Jarn\'ik--Besicovitch theorem should look like, we will comment on a well-known example. Let $\HD(S)$ denote the Hausdorff dimension of the set $S$.
\begin{theorem}[{\cite[Corollary 2]{BDV2}} and {\cite[Theorem 1]{DickinsonDodson}}]
Let $M$ be the unit circle in $\R^2$. Fix $c > 1/2$, and let $\psi_c(q) = q^{-c}$. Then
\begin{equation}
\label{jarnik}
\HD(W_{\psi_c}) = \begin{cases}
\frac{2 - c}{1 + c} & c \leq 1\\
\frac{1}{1 + c} & c \geq 1
\end{cases}.
\end{equation}
\end{theorem}
In fact, the ``phase transition'' which occurs here at $c = 1$ is due to a difference between intrinsic and extrinsic approximation. Specifically, we have the following:
\begin{theorem}
Let $M$ be the unit circle in $\R^2$. Fix $c > 0$, and let $\psi_c(q) = q^{-c}$. Then
\begin{equation}
\label{jarnikint}
\HD(W_{\psi_c}^\int) = \frac{1}{1 + c}
\end{equation}
while
\begin{equation}
\label{jarnikext}
\HD(W_{\psi_c}^\ext) = \begin{cases}
1 & c \leq 1/2\\
\frac{2 - c}{1 + c} & 1/2 \leq c < 1\\
0 & c > 1
\end{cases}.
\end{equation}
\end{theorem}
\begin{proof}
\eqref{jarnikint} is proven for example in \cite[Theorem 2.13]{FKMS}. When $c < 1$, \eqref{jarnikext} follows immediately from \eqref{jarnik} and \eqref{jarnikint}, since $\HD(W_{\psi_c}) = \max(\HD(W_{\psi_c}^\int),\HD(W_{\psi_c}^\ext))$. Finally, if $c > 1$ then \eqref{jarnikext} is a consequence of \cite[Lemma 1]{DickinsonDodson}.
\end{proof}
In particular, from the above example we see that the Hausdorff dimensions of $W_{\psi_c}$ and $W_{\psi_c}^\ext$ do not agree if $c$ is large enough. Thus for these values of $c$, an extrinsic analogue of the Jarn\'ik--Besicovitch theorem \emph{could not} be deduced directly from an ambient analogue.

\subsection{Open questions}
\label{subsectionopenquestions}

\begin{openquestion}
What is the correct generalization of Proposition \ref{propositionvonkoch}? More precisely, find a class of iterated function systems with the following properties:
\begin{itemize}
\item[(i)] It's easy to check whether or not any given IFS is in the class.
\item[(ii)] The von Koch snowflake is in the class.
\item[(iii)] No member of the class contains a line segment.
\end{itemize}
\end{openquestion}

\begin{openquestion}
Find an extrinsic analogue of the Jarn\'ik--Besicovitch theorem for some class of manifolds or fractals.
\end{openquestion}

\bibliographystyle{amsplain}

\bibliography{bibliography}

\end{document}